\documentclass[12pt, reqno]{amsart}
\usepackage{amsmath, amsthm, amscd, amsfonts, amssymb, graphicx, color}
\usepackage[bookmarksnumbered, colorlinks, plainpages]{hyperref}

\makeatletter \oddsidemargin.9375in \evensidemargin \oddsidemargin
\newtheorem{theorem}{Theorem}[section]
\newtheorem{lemma}[theorem]{Lemma}

\theoremstyle{definition}

\newtheorem{question}[theorem]{Question}
\newtheorem{example}[theorem]{Example}

\theoremstyle{remark}
\newtheorem{remark}[theorem]{Remark}
\numberwithin{equation}{section}
\begin{document}
\setcounter{page}{1}


\title[CONVEX-cyclic W. T. ON L. C. GROUPS]
{CONVEX-cyclic WEIGHTED TRANSLATIONS ON LOCALLY COMPACT GROUPS}

\author[M. R. Azimi, I. Akbarbaglu and M. Asadipour]
{M. R. Azimi,  I. Akbarbaglu and M. Asadipour}

\address {M. R. Azimi}
\address{Department of Mathematics, Faculty of Sciences,
University of Maragheh, 55181-83111, Golshahr, Maragheh, Iran}
\email{mhr.azimi@maragheh.ac.ir}
\address{I. Akbarbaglu}
\address{Department of Mathematics, Farhangian University,
  Tehran, Iran}
\email{ibrahim.akbarbaglu@gmail.com;i.akbarbaglu@cfu.ac.ir}
\address {M. Asadipour}
\address{Department of Mathematics, College of Sciences,
Yasouj University, Yasouj, 75918-74934, Iran}
\email{asadipour@yu.ac.ir}
\begin{abstract}
A bounded linear operator $T$ on a Banach space $X$ is called a
convex-cyclic operator if there exists a vector $x \in X$ such that
the convex hull of $Orb(T, x)$ is dense in $X$.
In this paper, for given an aperiodic element $g$ in a locally compact group $G$,
we give some sufficient conditions for
a weighted translation operator $T_{g,w}: f \mapsto w\cdot f*\delta_g$ on
$\mathfrak{L}^{p}(G)$ to be convex-cyclic. A necessary
condition  is also studied. At the end, to explain the obtained
results, some examples are given.
\end{abstract}

\maketitle


\section{Introduction}

A bounded linear operator $T$ on a separable infinite-dimensional
Banach space $X$ over the field $\mathbb{C}$ is called
\emph{hypercyclic} and supercyclic if there exists a vector $x\in X$
such that $\overline{Orb(T , x)}=X$ and $\overline{\mathbb{C}. Orb(T
, x)}=X$, respectively, where $Orb(T , x)=\{T^{n}x; \;\; n\in
\mathbb{N}  \}$. If $span\big(Orb(T , x)\big)$ is dense in $X$,
then $T$ is called a cyclic operator. Recall that, the
notion of hypercyclicity was already studied by Birkhoff \cite{Bir}
when he introduced the notion of the topological transitivity. To be
precise, an operator $ T $ is \emph{topologically transitive}, if
for every pair of nonempty open subsets $ U, V $ of $ X $, there
exists a non-negative integer $ n $ such that $ T^{n}(U)\cap V
\not=\emptyset $. It is not difficult to observe that
$$ Transitivity \Longleftrightarrow Hypercyclicity \Longrightarrow Supercyclicity \Longrightarrow Cyclicity.$$

Similar to the definition of the transitivity, if there exists a
non-negative integer $ N $ such that $ T^{n}(U)\cap V \not=\emptyset$
for all integers $ n\geq N $, then $T$ is called
\emph{topologically mixing} and it is clear that
$$Mixing \Longrightarrow Transitivity.$$



Like supercyclicity, another well known concept can be appeared
between cyclicity and hypercyclicity that is when the convex hull
generated by an orbit $Orb(T , x)$ is dense in $X$. In this case,
$x$ is called a convex-cyclic vector for $T$ and also, $T$ is called
a \emph{convex-cyclic} operator. Note that, if $\mathfrak{Cp}$
denotes the set of all convex polynomials, then $co(Orb(T, x)) =
\{P(T)x : P \in \mathfrak{Cp}\}.$ Similar to the definition of the
transitivity, a bounded linear operator $T$ on $X$ is called a
\emph{convex-transitive} operator if for every nonempty open subsets
$V$ and $U$ of $X$, the intersection $P(T)(V)\cap U$ is nonempty for
some convex polynomial $P \in \mathfrak{Cp}$. The relations between
convex-cyclicity and convex-transitivity is exhibited in the
following diagram.
$$Convex-transitivity \Rightarrow Convex-cyclicity.$$
The converse of the above diagram is correct
when the point spectrum of the adjoint of the associated operator is empty \cite[Theorem 3.9]{Re}.
Initially, the notion of the convex-cyclicity has been studied by H.
Rezaei in \cite{Re}. After that, other authors studied this notion in
\cite{Feldman}, \cite{Saavedra} and \cite{Sazegar}. For
example, the Hahn-Banach characterization for convex-cyclicity
theorem is an important theorem which we will use it in the
next section. Therefore we state it below and of course, its proof
can be found in \cite{Feldman}:\\
\textbf{\textit{Hahn-Banach Characterization for Convex-Cyclicity;}}
Let $X$ be a Banach space over the complex numbers, $T$ be a bounded linear operator on $X$ and $x \in X$.
Then the following are equivalent:\\
\textit{i}) $x$ is a convex-cyclic vector for $T$.\\
\textit{ii}) $\sup \limits_{n\geq 0} Re(\Lambda (T^{n}x)) = +\infty$
for every linear functional $ \Lambda \in X^*\setminus \lbrace 0\rbrace$.

Let us mention that for an arbitrary bounded sequence $w=\{w_{j}> 0\}_{j\in \mathbb{N}}$
and the canonical basis $\{e_{j}\}_{j\in \mathbb{N}}$ of $\ell_{p}(\mathbb{N})$ when $p\in [1,+\infty)$,
the operator $B_{w}$ on $\ell_{p}(\mathbb{N})$,
which is defined by
$$ B_{w}(e_{j})=
\begin{cases}
w_{j}e_{j-1}, &  j \geq 2\hspace{3mm} \\
0 &  j=1\hspace{3mm}
\end{cases}, $$
is called a unilateral weighted backward shift.
First of all, Rolewicz \cite{Rolewicz} showed that $\lambda B$ on $\ell_{2}(\mathbb{N})$ is
hypercyclic for any complex number $\vert\lambda\vert>1$
while the backward shift operator $B$ is not hypercyclic.
Then for $p\in [1,\;+\infty)$ and an arbitrary bounded sequence $w=\{w_{j}> 0\}_{j\in \mathbb{Z}}$
and the standard base $\{e_{j}\}_{j\in \mathbb{Z}}$ for $\ell_{p}(\mathbb{Z})$
the hypercyclic bilateral weighted shift operators on $\ell_{p}(\mathbb{Z})$,
which is defined by $ B_{w}(e_{j})=w_{j}e_{j-1}, $ have been characterized
by Salas \cite{Salas} in terms of their weights.\\
In \cite{Re}, Rezaei characterized a convex-cyclic weighted backward shift on $ \ell_{p}(\mathbb{N}) $.
Indeed, he proved that a weighted backward shift on $ \ell_{p}(\mathbb{N}) $
is hypercyclic if and only if it is convex-cyclic.
Then he claimed that a certain condition is sufficient for a bilateral weighted shift on
$ \ell_{p}(\mathbb{Z}) $ to be convex-cyclic \cite[Theorem 4.2]{Re},
but the proof of the associate theorem is correct only in a special case.
So in this paper, we will give a full generalization of this result to
weighted translation operators on the Lebesgue spaces.\\
In \cite{Chen1} and \cite{Chen2} the hypercyclic weighted translation operators
on the Lebesgue space $\mathfrak{L}^{p}(G)$
in terms of the weight were characterized in discrete and non-discrete groups.
Therefore, it is natural to raise the following question:
\begin{question}
Is there a description of the convex-transitive weighted translation operators on
the Lebesgue space $\mathfrak{L}^{p}(G)$?
\end{question}
For deep understanding of the hypercyclicity, we would like to mention two excellent books \cite{B} and \cite{G}.
Also, dynamics of the weighted translations in different settings have
been studied in \cite{ak, kum, Azadikhouy, az, az2}.
\section{Main results}

The operators that we want to study in this paper, are
in the complex Lebesgue space $\mathfrak{L}^{p}(G)$, when
$G$ is a locally compact group with a right Haar measure
$\vartheta$. Note that, the Lebesgue space $\mathfrak{L}^{p}(G)$
with recent properties is separable whenever $G$
is second countable.
We remind that a torsion element in a locally compact group $G$ is an element of finite order,
and an element $g\in G$ is called compact \cite{Hewitt} (or periodic \cite{Grosser})
if the closed subgroup $G(g)$ generated by $g$ is compact. Also, an element in
$G$ is called aperiodic if it is not periodic.
Observe that in discrete groups, a periodic element is a torsion element and conversely.\\
Given an element $g \in G$, the unit point mass function
at $g$ is denoted by $\delta_{g}$, and a weight on $G$
is considered as a real bounded function $w : G \rightarrow (0,+\infty)$.
Then a weighted translation operator $T_{g,w}$ is defined on
$\mathfrak{L}^{p}(G)$ by
$$T_{g,w}(f)(x):=w(x)(f*\delta_{g})(x)=
w(x) \int_{G}f(xh^{-1})d\delta_{g}(h)=
w(x) f(xg^{-1}).$$
It is not difficult to observe that,
\begin{equation} \label{eq}
T^{n}_{g,w}(f)=\bigg( \prod\limits_{i=0}^{n-1}w*\delta_{g^{i}}\bigg)(f*\delta_{g^{n}}).
\end{equation}
Chen and Chu \cite{Chen2} were succeeded in showing how topologically
mixing of $T_{g,w}$ depends on the behavior of successive
translations of $w$ by $g$. More precisely, if
$T_{g,w}:\mathfrak{L}^{p}(G)\rightarrow \mathfrak{L}^{p}(G)$ is a
weighted translation,
then the following statements are equivalent (\cite[Theorem 2.2]{Chen2}).\\
\textit{i)} $T_{g,w}$ is a topologically mixing weighted translation.\\
\textit{ii)} If $K$ is an arbitrary compact subset of $G$ with
$\vartheta(K)>0$, then there exists a sequence of Borel sets
$\{B_{n}\}\subset K$ such that $\displaystyle\lim_{n\rightarrow
\infty}\vartheta(B_{n})=\vartheta(K)$ and both sequences $\{w_{n}\}$
and $\{\tilde{ w}_{n}\}$ which
$$ w_{n}:=\prod_{i=1}^{n}w*\delta_{g^{-1}}^{i}\;\;\;\;and\;\;\;
\tilde{w}_{n}:=\dfrac{1}{ \prod_{i=0}^{n-1} w*\delta_{g}^{i}}  $$
satisfy
$$\lim_{n\rightarrow \infty} \Vert w_{n}\vert_{B_{n}} \Vert_{\infty}=
\lim_{n\rightarrow \infty} \Vert \tilde{w}_{n}\vert_{B_{n}} \Vert_{\infty}=0.$$
We are now ready to show how the convex-cyclicity of the weighted
translation $T_{g,w}$ depends on the behavior of successive
translations of the weight $w$ by an aperiodic element $g$. Since we
will use the convex-transitive criterion, so it is stated \cite{Re}.\\
\textbf{\textit{Convex-transitive Criterion Theorem.}}
Let $X$ be a separable Banach space and $T$ be a bounded
linear operator on $X$. If there exist dense subsets $Y$ and $Z$ of
$X$ such that for every vectors $y, z $ in $Y$ and $Z$,
respectively, there exist a sequence $\{P_{k}\}_{k>1}\subseteq
\mathfrak{Cp}$
and functions $S_{k}: Z \rightarrow X$ such that\\
\textit{i}) $P_{k}(T)y\rightarrow 0$,\\
\textit{ii}) $S_{k}z \rightarrow0$ and $P_{k}(T)S_{k}z\rightarrow z,$\\
then $T$ is convex-transitive and consequently is convex-cyclic.






In the following theorem, we are going to state the sufficient
conditions for the weighted translation operator $T_{g,w}$ to be
convex-cyclic. This result not only extends \cite[Theorem 4.2]{Re} but also improves its proof.\\
In what follows, to prevent the text from prolongation, we use the following symbols.\\
\textit{i)} $\mathfrak{AP}(G)$ is instead of all aperiodic elements of $G$.\\
\textit{ii)} The set of all bounded positive functions
$w : G \rightarrow (0,+\infty)$ will be denoted by $\Psi(G)$.

\begin{theorem}\label{Th A}
Let $g \in \mathfrak{AP}(G)$ and $w \in\Psi(G)$. Then the weighted
translation operator $T_{g,w}:\mathfrak{L}^{p}(G)\rightarrow
\mathfrak{L}^{p}(G)$ is convex-cyclic whenever there exists a scalar
$\beta \geqslant 1$ such that
\begin{itemize}
\item [(i)] $\liminf \limits_{k \rightarrow\infty} \Vert \beta^{-k}
\prod\limits_{i=1}^{k} w*\delta_{g^{-i}} \Vert_{\infty}=0, \quad
\liminf \limits_{k \rightarrow\infty}
\Vert\beta^{k} \prod\limits_{i=0}^{k-1} w*\delta_{g^{i}})^{-1}
\Vert_{\infty}=0$\\ whenever $\beta >1$ and
\item [(ii)] $\liminf \limits_{k \rightarrow\infty} \Vert k^{-1}
\prod\limits_{i=1}^{k} w*\delta_{g^{-i}} \Vert_{\infty}=0, \quad
\liminf \limits_{k \rightarrow\infty}
\Vert k\prod\limits_{i=0}^{k-1} w*\delta_{g^{i}} )^{-1} \Vert_{\infty}
=0 $\\ whenever $\beta =1$.
\end{itemize}
\end{theorem}

\begin{proof}
The condition \textit{(i)} implies that $T_{g,\beta^{-1} w}=\beta^{-1}T_{g,w}$
is hypercyclic, \cite[Theorem 2.3]{Chen2}. Hence the set $\mathcal{H}$ consisting of all
hypercyclic vectors for $T_{g, \beta^{-1}w}$ are dense in $\mathfrak{L}^{p}(G)$.
Also, the condition \textit{(i)} implies that $ker(\beta I-T_{g,w})=\{0\}$ and
hence $(\beta I-T_{g,w})$ has dense range.

Now we want to show that $T_{g, w}$ satisfies the convex-cyclic criterion.
For this goal, we put $Y=(\beta I-T_{g,w})(\mathcal{H})$ and
$Z=C_c(G)$, the set of all continuous
$f \in \mathfrak{L}^{p}(G)$ with compact support. Thus $Y$ and $Z$ are dense subsets of $\mathfrak{L}^{p}(G)$.

Now we consider a convex polynomial
$$P_{k}(t)=\dfrac{\beta-1}{\beta^{k}-1} (\beta^{k-1}+\beta^{k-2}t+\cdots + t^{k-1}), \;\;\; k\in \mathbb{N},$$
and if we define
$$S_{g,w}(f)=(\dfrac{1}{w}f)*\delta_{g^{-1}}, \;\;\; f \in \mathfrak{L}^{p}(G), $$
then set
\begin{equation} \label{eq.}
S_{k}:= \dfrac{\beta^{k}-1}{1-\beta}(\beta I-T_{g, w})S^{k}_{g,w}.
\end{equation}
For an arbitrary $f_{0} \in \mathcal{H}$, we have $(\beta I-T_{g, w})f_{0} \in Y$ 
and hence, 
\begin{align*} &P_{k}(T_{g, w})(\beta I-T_{g,
w})(f_{0})= \dfrac{\beta-1}{\beta^{k}-1} \big(\beta^{k}f_{0}-
T^{k}_{g,w}(f_{0})\big)\\&=\dfrac{\beta-1}{\beta^{k}-1} \big(\beta^{k}f_{0}-
 (\prod_{i=0}^{k-1}w*\delta_{g^{i}})
f_{0}*\delta_{g^{k}}\big)\\&= \dfrac{\beta^{k}(\beta-1)}{\beta^{k}-1}
\big(f_{0}- (\beta^{-k}
\prod_{i=0}^{k-1}w*\delta_{g^{i}})f_{0}*\delta_{g^{k}}\big)\\&=
\dfrac{\beta^{k}(\beta-1)}{\beta^{k}-1} \big(f_{0}-T^{k}_{g,
\beta^{-1}w}(f_{0})\big).
\end{align*}
Now, above equality implies that
$$\Vert P_{k}(T_{g, w})(\beta I-T_{g, w})(f_{0}) \Vert_{p}=
\Vert \dfrac{\beta^{k}(\beta-1)}{\beta^{k}-1} \big( f_{0}-T^{k}_{g,
\beta^{-1}w}(f_{0}) \big) \Vert_{p} \rightarrow 0 $$ 
as $k\rightarrow \infty$. On the other side, the definition of $S_{g,w}$
implies that
$$ S^{k}_{g,w}(h)=\big(\prod_{i=1}^{k}w*\delta_{g^{-i}}\big)^{-1} (h*\delta_{g^{-k}})
, \;\;\; h \in C_c(G).$$ Now, from \eqref{eq.} we get that
$$ \Vert S_{k}(h) \Vert_{p}\leqslant \dfrac{1}{\beta -1} ( \beta +\Vert T_{g,w }\Vert)
\Vert \beta^{k} ( \prod_{i=0}^{k-1}w*\delta_{g^{i}})^{-1}
\Vert_{\infty} \Vert h \Vert_{p},$$ so $ \Vert S_{k}(h) \Vert_{p}
\rightarrow 0 $ as $k \rightarrow \infty$. Moreover, for every $h\in C_c(G)$ we have
$$\Vert P_{k}(T_{g,w})S_{k}(h)-h \Vert_{p}=\Vert \beta^{k} S^{k}_{g,w}(h) \Vert_{p} \leqslant
\Vert \beta^{k} ( \prod_{i=0}^{k-1}w*\delta_{g^{i}})^{-1}
\Vert_{\infty} \Vert h \Vert_{p}.$$ Hence, the condition
\textit{(i)} and the above inequality implies
$P_{k}(T_{g,w})S_{k}(h) $ tends to $h$ in $\Vert \cdot \Vert_{p}-$norm,
whenever $k\rightarrow \infty$.
Therefore $T_{g,w}$ satisfies the convex-transitive criterion and consequently it is convex-cyclic.\\
When the condition \textit{(ii)} holds, then we may consider the convex polynomial
$$P_{k}(t)=k^{-1} (1+t+t^{2}+\cdots + t^{k-1}),$$
and 
$$
S_k:=k(I-T_{g,w})S^{k}_{g,w}
$$
for any $k\in \mathbb{N}$.
The condition \textit{(ii)}
implies that $ker(I-T_{g,w})=\{0\}$ and hence the operator
$(I-T_{g,w})$ has dense range. Indeed, if there exists a non-zero $h\in
\mathfrak{L}^{p}(G)$ such that $T_{g,w}(h)=h$, then since
$T_{g,w}$ is invertible and $T^{-1}_{g,w}= T_{g^{-1},(\frac{1}{w}*\delta_{g^{-1}})}$,
we obtain
$$ h=T^{k}_{g^{-1},(\frac{1}{w}*\delta_{g^{-1}})}(h)=
\big(\prod_{i=1}^{k}w*\delta_{g^{-i}}\big)^{-1} (h*\delta_{g^{-k}}), \;\;\; k
\in \mathbb{N}.$$ Now, the inequality
\begin{align*}
\Vert h \Vert^{p}_{p}&= \int_{G} \vert (
\prod_{i=1}^{k}w*\delta_{g^{-i}})^{-1}(x) \vert^{p} \vert (
h*\delta_{g^{-k}})(x) \vert^{p} d\vartheta(x)
\\&= \int_{G} \vert (
\prod_{i=0}^{k-1}w*\delta_{g^{i}})^{-1}(x) \vert^{p} \vert h(x)
\vert^{p} d\vartheta(x)\\& \leqslant \Vert k (
\prod_{i=0}^{k-1}w*\delta_{g^{i}})^{-1} \Vert^{p}_{\infty} \Vert h
\Vert^{p}_{p}
\end{align*}
implies that $ker(I-T_{g,w})=\{0\}$. \\
Now set $Y=Z=C_c(G)$. Then for each $f\in C_c(G)$ we have
\begin{align*}
&\Vert P_{k}(T_{g, w})(I-T_{g, w})(f) \Vert_{p}=
\Vert \frac{1}{k}f-\frac{1}{k}(\prod_{i=0}^{k-1}w*\delta_g^{i})f*\delta_{g^k} \Vert_{p} \\
&=\Vert\frac{1}{k}f*\delta_g^{-k}-\frac{1}{k}(\prod_{i=1}^{k}w*\delta_{g^{-i}})f\Vert_{p}\\
&\leqslant\frac{1}{k}\Vert f*\delta_{g^{-k}}\Vert_{p}+
\Vert\frac{1}{k}\prod_{i=1}^{k}w*\delta_{g^{-i}}\Vert_{\infty}\Vert f \Vert_{p}\rightarrow 0 
\end{align*}
as $k\to\infty$. Similarly
$$ \Vert S_k(f) \Vert_{p}\leqslant (1+\Vert T_{g,w} \Vert)
\Vert k(\prod_{i=0}^{k-1}w*\delta_{g^i})^{-1}) \Vert_{\infty}\Vert f \Vert_{p}$$
and
$$ \Vert P_{k}(T_{g,w})S_{k}(f)-f \Vert_{p}=\Vert S^{k}_{g,w}(f) \Vert_{p} \leqslant
\Vert k(\prod_{i=0}^{k-1}w*\delta_{g^{i}})^{-1}\Vert_{\infty} \Vert f \Vert_{p}, $$
that yields $S_k(f)$ and $P_k(T_{g,w})S_k(f)$ approaches to $0$ and $f$, respectively, as $k\to\infty$.
Therefore, $T_{g,w}$ is again convex-cyclic in this case.
\end{proof}

We would like to state the following lemma because we will use it in the proof of the next Theorem.
\begin{lemma}\label{lem1}
Let $x_0\in X$ be a convex-cyclic vector for an operator $T$ on $X$,
$ \varepsilon >0 $ and $N_0 \in \mathbb{N}$. Then there is a
convex polynomial $P_{k}(t):= a_{0}+ a_{1}t+\cdots + a_{k}t^{k}$,
$a_{k}>0$ such that $k>N_0$ and $ \Vert P_{k}(T)(x_0)-x  \Vert < \varepsilon $ 
for every $x\in X$.
\end{lemma}
\begin{proof}
Since $x_0\in X$ is a convex-cyclic vector for $T$, so there is a convex
polynomial $P_{n}(t):= \sum\limits_{i=0}^{n}a_{i}t^{i}$, $a_{n}>0$
such that $ \Vert P_{n}(T)(x_0) \Vert < \dfrac{\varepsilon}{ \Vert T \Vert^{N_0} } $ and consequently
\begin{align} \label{A}
\Vert T^{N_0}\big( a_{0}x_0+ a_{1}T(x_0)+\cdots + a_{n}T^{n}(x_0) \big)  \Vert <\varepsilon.
\end{align}
Now for an arbitrary vector $x\in X$, there is a convex
polynomial
 $Q_{m}(t):= \sum\limits_{i=0}^{m}b_{i}t^{i} $, $b_{m}>0$
 such that $ \Vert Q_{m}(T)(x_0) -2x \Vert < \varepsilon. $
 Thus
 \begin{align} \label{B}
\Vert \dfrac{1}{2}Q_{m}(T)(x_0) -x \Vert < \dfrac{\varepsilon}{2}.
    \end{align}
From \eqref{A} and \eqref{B} we get that
$$ \Vert\dfrac{1}{2} \big( T^{N_0}P_{n}(T)(x_0)+Q_{m}(T)(x_0) \big)-x \Vert < \varepsilon .  $$
Note that the polynomial $ \dfrac{1}{2} \big(T^{N_0} P_{n}(T)x+Q_{m}(T)x
\big)$ is a convex polynomial such that its degree is equal to $l:=
\max\{ n+N_0 , m  \} $. Therefore $l >N_{0}$ and the proof is completed.
\end{proof}

\begin{theorem} \label{Th B}
Let $g \in \mathfrak{AP}(G)$, $ w \in\Psi(G)$ and
also, let the weighted translation operator
$T_{g,w}$ on $\mathfrak{L}^{p}(G)$ be convex-cyclic. If
$\sigma_{p}(T_{g,w}^{*})=\emptyset$, then for each compact subset
$K\subset G$ with positive measure, there exist $N_0\in
\mathbb{N}$, a sequence of Borel subsets $(E_n)\subseteq K$ and  a
 sequence $(a_n)\subseteq [0,1]$ with $\sum_{i=0}^{\infty}a_i=1$ such that
\begin{align*}
\small 1/ \liminf_{n\rightarrow\infty}
\Vert(a_{0}\prod_{i=0}^{N_0-1}w*\delta_{g^{i}}
+\cdots +
a_{n}\prod\limits_{i=0}^{N_0+n-1} w*\delta_{g^{i}}
)|_{E_n}\Vert_{\infty}=0.
\end{align*}
 \end{theorem}

\begin{proof}
Let $\varepsilon>0$ and $K\subset G$ be an arbitrary compact set
with $\vartheta(K)>0$. Then there exists a positive integer $N_{0}$
such that for every $n>N_{0}$,  $K\cap Kg^{n}=\emptyset$, because
$g$ is aperiodic \cite[Lemma 2.1]{Chen2}. Note that, if
$\sigma_{p}(T_{g,w}^{*})=\emptyset$, then the set of all
convex-cyclic vectors for $T_{g,w}$ is dense in
$\mathfrak{L}^{p}(G)$ \cite[Theorem 3.5]{Re}. Consider $\chi_{K}\in
\mathfrak{L}^{p}(G)$ as the characteristic function of $K$ and
choose $\theta \in (0,\; \dfrac{\varepsilon}{1+\varepsilon})$, then
there is a convex-cyclic vector $h\in \mathfrak{L}^{p}(G)$ such
that $\Vert h- \chi_{K}  \Vert_{p}<\theta^{2}$. Also there exists a
convex polynomial $P_{n}\in \mathfrak{Cp}$ with the real
coefficients of the finitely supported sequence $(a_n)\subseteq
[0,1]$  such that $a_0+a_1+ \cdot\cdot\cdot +a_n=1$ and 
$$\Vert P_{n}(T_{g,w})(h)- \chi_{K} \Vert_{p}<\theta^{2}.$$

 To attain the
desired  result, by Lemma \ref{lem1}, the power of the polynomial $P_n$ is somehow chosen
greater than $N_0$ and sufficiently large as well.
Furthermore, by \cite[Corollary 2.2]{Feldman}, it does not make an ambiguity if
we start the first term of $P_n$ of order $N_0$ i.e., zero is the
root of the order $N_0$. Indeed,
$P_n(t)=a_0t^{N_0}+a_1t^{N_0+1}+\cdots +a_nt^{N_0+n}$.  Now,
consider the following Borel subsets of $G$.
\begin{align*}
B_{\theta}&=\{x\in G\setminus K:\;\; |h(x)| \geqslant \theta \},\\
C_{\theta}&=\{x\in K:\;\; \vert
\big(a_{0}(\prod\limits_{i=0}^{N_0-1}w*\delta_{g^{i}})
h*\delta_{g^{N_0}}+a_{1}(\prod\limits_{i=0}^{N_0} w*\delta_{g^{i}})
h*\delta_{g^{N_0+1}}\\ &+\cdots + a_{n}(\prod\limits_{i=0}^{N_0+n-1}
w*\delta_{g^{i}}) h*\delta_{g^{N_0+n}}\big)(x)-1\vert \geqslant \theta
\}.
\end{align*}
Note that
\begin{align*}
\theta^{2p}&>\Vert h - \chi_{K} \Vert_{p}^p>\int_{G\setminus K} |h-
\chi_{K}|^p
d\vartheta\\
&\geq \int_{B_{\theta}} |h|^p d\vartheta\geq \theta^p
\vartheta(B_{\theta}).
\end{align*}
Moreover observe that,
\begin{align*}
\theta^{2p} & > \Vert  P_{n}(T_{g,w})(h)- \chi_{K}\Vert
_{p}^{p}\\
& =\int_{G} \vert
a_{0}(\prod\limits_{i=0}^{N_0-1}w*\delta_{g^{i}})
h*\delta_{g^{N_0}}+a_{1}(\prod\limits_{i=0}^{N_0} w*\delta_{g^{i}})
h*\delta_{g^{N_0+1}}\\ &+\cdots + a_{n}(\prod\limits_{i=0}^{N_0+n-1}
w*\delta_{g^{i}}) h*\delta_{g^{N_0+n}} -\chi_{K}\vert^{p}d\vartheta\\
&\geqslant \int_{K} \vert
a_{0}(\prod\limits_{i=0}^{N_0-1}w*\delta_{g^{i}})
h*\delta_{g^{N_0}}+a_{1}(\prod\limits_{i=0}^{N_0} w*\delta_{g^{i}})
h*\delta_{g^{N_0+1}}\\ &+\cdots + a_{n}(\prod\limits_{i=0}^{N_0+n-1}
w*\delta_{g^{i}}) h*\delta_{g^{N_0+n}}-1\vert^{p}d\vartheta\\
&\geqslant \theta^{p} \vartheta(C_\theta),
\end{align*}
which implies that $\vartheta(C_\theta) \leqslant \theta^{p}$.
Furthermore, on the subset $K\setminus C_{\theta}$ we have
\begin{align*}
& 1-\theta < \vert a_{0}(\prod\limits_{i=0}^{N_0-1}w*\delta_{g^{i}})
h*\delta_{g^{N_0}}+a_{1}(\prod\limits_{i=0}^{N_0} w*\delta_{g^{i}})
h*\delta_{g^{N_0+1}}\\ &+\cdots + a_{n}(\prod\limits_{i=0}^{N_0+n-1}
w*\delta_{g^{i}}) h*\delta_{g^{N_0+n}}\vert\\
& \leqslant  \big(a_{0}\prod\limits_{i=0}^{N_0-1}w*\delta_{g^{i}}
+a_{1}\prod\limits_{i=0}^{N_0} w*\delta_{g^{i}} +\cdots\\ & +
a_{n}\prod\limits_{i=0}^{N_0+n-1} w*\delta_{g^{i}}\big)\max
\limits_{N_0\leqslant i \leqslant N_0+n-1} |h*\delta_{g^{i}}|.
\end{align*}
On the other hand, according the fact that  $K\cap Kg^{n}=\emptyset$, then on the subset
$E_n:=K\setminus
 (B_{\theta}g^{N_0} \cup B_{\theta}g^{N_0+1} \cup \cdots \cup B_{\theta}g^{N_0+n}
  \cup C_{\theta})$ we obtain
\begin{align*}
&1/ (a_{0}\prod\limits_{i=0}^{N_0-1}w*\delta_{g^{i}}
+a_{1}\prod\limits_{i=0}^{N_0} w*\delta_{g^{i}} +\cdots +
a_{n}\prod\limits_{i=0}^{N_0+n-1}
w*\delta_{g^{i}}) \\
&< \frac{1}{1-\theta}\max\limits_{N_0 \leqslant i \leqslant
N_0+n}\vert h*\delta_{g^{i}}\vert\\&<
\dfrac{\theta}{1-\theta}\\&<\varepsilon,
\end{align*}
which completes the proof.
\end{proof}
\begin{remark}
If one sets $w=1$ in the statement of Theorem \ref{Th B}, it is
pointed out that the translation operator $T_{g}$ can not be
convex-cyclic itself.
\end{remark}
\begin{example}
Let $G=\mathbb{R}$, be the group of the real numbers
equipped with
the Lebesgue measure. Fix a non-zero negative $g\in \mathbb{R}$ and consider
the convolution  $(f*\delta_{g})(x)=f(x-g)$, $x\in \mathbb{R}$ and $f\in
\mathfrak{L}^p(\mathbb{R})$.  Choose  arbitrary real numbers $s,t$
as $1<s<t$, then define the weight function $w$ on $\mathbb{R}$ by
$$w(x)=\left\{
  \begin{array}{ll}
    t, & {1\leq x}, \\
        -\frac{x}{2}+1, & {-1< x < 1},\\
    s, & {x\leq -1}.
  \end{array}
\right.$$  By taking any $\beta\in \mathbb{R}$ with $s<\beta <t$,
Theorem \ref{Th A} ensures that $T_{g,w}$ is  convex-cyclic on
$\mathfrak{L}^p(\mathbb{R})$  while is not hypercyclic by
\cite[Theorem 2.2]{Chen2}.
 \end{example}

\begin{example}
Observe that by the second case
in Theorem \ref{Th A}, the weighted translation operator
$T:=T_{g,(w*\delta_{g})} $ on $\mathfrak{L}^{p}(\mathbb{Z})$ is a
convex-cyclic operator whenever $g=-1$ and
$$w=\lbrace w_{i} \rbrace= \begin{cases}
2,  &  i\geqslant 1 \\
1,  &  i \leqslant 0.
\end{cases} $$
We show that the point spectrum of the
adjoint of $T$ is empty. Indeed, if $\lambda \in \sigma_{p}(T^{*})$,
then there is a non-zero $f \in \mathfrak{L}^{q}(\mathbb{Z}) $ such
that $T^{*}f= \lambda f$. Since $(T^{*})^{n}f= \lambda^{n} f$ for
every $n \in \mathbb{N}$, so $ \prod
\limits_{i=0}^{n-1}w(m-i)f(m-n)=\lambda^{n}f(m) $ for every $m \in
\mathbb{Z}$. If $m=0$, then $f(-n)=\lambda^{n}f(0)$. This implies
that $ \vert \lambda \vert < 1$ because $|f(-n)|\rightarrow 0$ as
$n\rightarrow \infty$. But if $m >0$, then $2^{m}f(0)=\lambda^{m}
f(m)$. Thus $2< \vert \lambda \vert$ because $|f(m)|\rightarrow 0$
as $n \rightarrow \infty$. This contradiction shows that
$\sigma_{p}(T^{*})=\phi$.\\
Now we consider $ P_{n}(x)= a_{0}x^{N_{0}}+a_{1}x^{N_{0}+1}+\cdots
+a_{n}x^{N_{0}+n} $ as a convex polynomial when $N_{0}\in
\mathbb{N}$ such that $N_{0}=\vartheta (K_{N_{0}})$ for an arbitrary
finite subset $K_{N_{0}}=\{ m+1, \; m+2,\; \cdots, \; m+N_{0} \}$ of
$\mathbb{Z}$. It is not difficult to see that $K_{N_{0}} \bigcap K
g^{\pm N_{0}} = \phi .$ For simplicity assume that all elements in $K_{N_0}$ are negative. 
If $0 < \theta < 1 $, then there exist an $h
\in \mathfrak{L}^{p}(\mathbb{Z})$ and a convex polynomial $P_{n}(x)$
similar above such that
$$ \big\Vert h-\chi_{K_{n_{0}}} \big\Vert_{p}\leqslant \theta^{2} \;\; \mbox{ and} \;\;\;\;
\big\Vert P_{n}(T)( h)-\chi_{K_{n_{0}}} \big\Vert_{p}\leqslant
\theta^{2}. $$ Let
\begin{align*}
A&= \{ x \in K_{N_{0}}: \vert h(x)-1 \vert \geqslant \theta \},\\
B&=\{ x \in \mathbb{Z} \setminus K_{N_{0}}: \vert h(x) \vert
\geqslant \theta \},\\
 C&= \{ x \in K_{N_{0}}: \big\vert a_{0}
\prod_{i=1}^{N_{0}}w(x+i) h(x+N_{0}) + \\&a_{1}
\prod_{i=1}^{N_{0}+1}w(x+i) h(x+N_{0}+1)+ \cdots +\\& a_{n} \prod
\limits _{i=1}^{N_{0}+n}w(x+i) h(x+N_{0}+n) -1
\big\vert \geqslant \theta \}\\
&\mbox{and} \\
D&= \{ x \in K_{N_{0}}: \big\vert a_{0}
\prod\limits_{i=1}^{N_{0}}w(x+i-n) h(x+N_{0}-n)+\\& a_{1} \prod
\limits_{i=1}^{N_{0}+1}w(x+i-n) h(x+N_{0}+1) + \cdots +\\& a_{n}
\prod \limits _{i=1}^{N_{0}+n}w(x+i-n) h(x+N_{0})
 \big\vert \geqslant \theta \}.
 \end{align*}
Then\; $\vartheta (A)<\theta^{p}$, $\vartheta (B)<\theta^{p}$,
$\vartheta (C)<\theta^{p}$ and $\vartheta (D)<\theta^{p}$ imply that
$\vartheta (A)=\vartheta (B)=\vartheta
 (C)=\vartheta (D)=0$.
Now, for an $x \in K_{N_{0}}$ we have that
$$a_{0} \prod_{i=1}^{N_{0}}w(x+i)+ a_{1} \prod_{i=1}^{N_{0}+1}w(x+i) +\cdots +
a_{n} \prod \limits _{i=1}^{N_{0}+n}w(x+i) \rightarrow \infty$$ as
$n \rightarrow \infty$, but
$$a_{0} \prod_{i=1}^{N_{0}}w(x+i-n)+ a_{1} \prod_{i=1}^{N_{0}+1}w(x+i-n) +\cdots +
a_{n} \prod \limits _{i=1}^{N_{0}+n}w(x+i-n)$$ $\rightarrow
\sum\limits_{j=0}^{n}a_{j}=1 $ as $n \rightarrow \infty$.
\end{example}

\bibliographystyle{amsplain}

\end{document}